\documentclass[a4paper,11pt]{article}

  \renewcommand\appendix{\par
  \setcounter{section}{0}
  \setcounter{sub
  section}{0}
  \setcounter{figure}{0}
  \setcounter{table}{0}
  \renewcommand\thesection{ Appendix \Alph{section}}
  \renewcommand\thefigure{\Alph{section}\arabic{figure}}
  \renewcommand\thetable{\Alph{section}\arabic{table}}
}

 \usepackage{amsmath, amsthm, amssymb}
\usepackage{amssymb}
\usepackage{graphicx}
\usepackage{algorithmic}
\usepackage{algorithm}

\usepackage{tikz}
\usetikzlibrary{calc}
\usetikzlibrary{patterns}
\tikzstyle{mybox} = [draw=black, fill=white,  thick,
    rectangle, inner sep=10pt, inner ysep=20pt]
\tikzstyle{mybox} = [draw=black, fill=white,  thick,
    rectangle, inner sep=2pt, inner ysep=2pt]

\usepackage[margin=1.2in]{geometry}

\usepackage{setspace}

\newtheorem{lemma}{Lemma}[section]
\newtheorem{theorem}{Theorem}[section]

\begin{document}

\title{Upper bounds for domination numbers of graphs using  Tur\'an's Theorem and  Lov\'asz local lemma}

\author{Sharareh Alipour, Amir Jafari}
\maketitle

\begin{abstract}
Let $G$ be a connected graph of order $n$ with vertex set $V(G)$. A subset $S\subseteq V(G)$ is an $(a,b)$-dominating set if every vertex  $v\in S$ is adjacent to at least $a$ vertices in $S$ and every $v\in V\setminus S$ is adjacent to at least $b$ vertices in $S$. The minimum cardinality of an $(a,b)$-dominating set of $G$ is the $(a,b)$-domination number of $G$, denoted by  $\gamma_{a,b}(G)$. There are various results about upper bounds  for $\gamma_{a,b}(G)$ when $G$ is regular or $a$ and $b$ are small numbers. 

In the first part of this paper, for a given graph $G$ with the minimum degree of $\max\{a,b\}$, 
we define a new graph $G'$ associated to $G$ and show that the independence number of this graph is related to $\gamma_{a,b}(G)$.
In the next part, using Lov\'asz local lemma, we give a randomized approach to improve previous results in some special cases.
 \end{abstract}

\section{Introduction}
\subsection{Problem statement}
In a graph $G$ with vertex set $V(G)$ and edge set $E(G)$, the open neighborhood of a vertex $v$ is $N(v) = \{u \in V(G):uv \in E(G)\}$. The degree of $v$, denoted by $d(v)$, is the cardinality of $N(v)$. Let $\delta(G)$ ($\Delta(G)$) be  the minimum (maximum) degree of vertices of $G$ and $G$ is $r$-regular if $d(v) = r$ for all $v \in V$ .

For positive integers $a$ and $b$, an $(a,b)$-dominating set of $G$ is a subset $S\subseteq V(G)$ such that every vertex  $v\in S$ is adjacent to at least $a$ vertices in $S$ and $v\in V(G)\setminus S$ is adjacent to at least $b$ vertices in $S$

The minimum cardinalities of $(a,b)$-dominating sets is denoted by $\gamma_{a,b}(G)$.
The special cases of when $(a,b)$ is one of the following pairs $(k,k)$, $(k-1,k)$ are respectively called $k$-tuple total, $k$-tuple  dominating numbers in the literature.

\subsection{Related works and our results}

Domination in graphs is now well studied in graph theory and the literature on this subject has been surveyed and detailed in the two books by Haynes, Hedetniemi, and Slater \cite{hed,sla}.
Dominating sets are of practical interest in several areas. 
The complexity of the domination problem also has been well-studied in the literature, see \cite{chang}. The hardness of approximation of the  domination problem has also been extensively investigated in the literature, see \cite{aus,lia,lia2,klas}.

While determining the exact value of $\gamma_{a,b}(G)$ is not easy, many studies focus on their upper bounds\cite{har2,klas,lia,lia2}. Here, we present some known upper bounds.

Let $G_{14}$ be the Heawood graph (or, equivalently, the incidence bipartite graph of the Fano plane) on $14$ vertices shown in Figure \ref{figh}. In \cite{hen}, Henning  and Yeo proved some theorems about strong transversal in hypergraphs and then as an application of their hypergraph results they proved the following theorem.
\begin{theorem}{\cite{hen}}
\label{1}
If $G \neq G_{14}$ is a connected graph of order $n$ with $\delta(G)\geq 3$, then $\gamma_{2,2}(G) \leq \frac{11}{13}n$, and $\gamma_{2,2}(G_{14})=12$.
\end{theorem}

\begin{figure}[h!]
\centering
  \includegraphics[width=50mm]{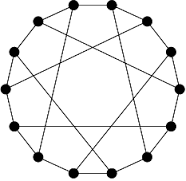}
  \caption{The Heawood graph, $G_{14}$.} 
 \label{figh}
\end{figure}

Now, let, $\tilde{d}_m=\frac{1}{n} \sum^n_{i=1} {d(v_i)+1\choose m}$. Then, we have the following theorem from \cite{chang2}.
\begin{theorem}{\cite{chang2}}
For any graph $G$ of minimum degree $\delta$ with $1 \le k\le \delta+ 1$

$$\gamma_{k-1,k}(G)\leq \frac{\ln(\delta-k+2)+\ln \tilde{d}_{k-1}+1}{\delta-k+2} n.$$
\label{11}
\end{theorem}

Also, let, $\hat{d}_m=\frac{1}{n} \sum^n_{i=1} {d_i\choose m}$, then we have the following theorem from \cite{kaz}.
\begin{theorem}{\cite{kaz}}
If $k$ is a positive integer and $G$ is a graph of order $n$ with $\delta > k \geq 1$, then
$$\gamma_{k,k}(G)\leq \frac{\ln(\delta-k)+\ln \hat{d}_{k}+1}{\delta-k} n.$$
\label{22}
\end{theorem}
Also, we have the following results from \cite{ali}.
\begin{theorem}
\label{a2}
Let $r\ge 3$. If $G$ is an $r$-regular graph of order $n$ which is not the incident graph of a projective plane of order $r-1$, then $\gamma_{r-1,r-1}(G) \leq \frac{r(r-1)-1}{r(r-1)}n$. If $G$ is the incidence graph of a projective plane of order $r-1$ then, $\gamma_{r-1,r-1}(G)=\frac{r(r-1)}{r(r-1)+1}n=2r(r-1)$. 
\end{theorem}

\begin{theorem}
\label{3}
If $G$ is an $r$-regular graph of order $n$ which is not a Moore graph of degree $r$ and diameter $2$ then $\gamma_{r-1,r}(G) \leq \frac{r^2-1}{r^2}n$, otherwise $\gamma_{r-1,r}(G)=\frac{r^2}{r^2+1}n=r^2$.
\end{theorem}
 
\section{Upper bound for $(a,b)$-domination number using  Tur\'an's Theorem}

In this section, we present our approach for computing an upper bound for $(a,b)$-domination number of a given graph $G$ with $\delta(G)\geq \max\{a,b\}$.
For a given graph $G$, we construct another graph $G'$ with the same set of vertices with the following property. For any independent set $A$ of vertices for $G'$, $V(G)\setminus A$ is an $(a,b)$-dominating set for $G$.
So, our goal is to compute a lower bound for  the independence number of $G'$. 
First, we recall  Tur\'an's Theorem.
\begin{theorem}
Let $G$ be any graph with $n$ vertices, such that $G$ is $K_{r+1}$ -free. Then, the number of edges in $G$ is at most
$(1-\frac{1}{r}).\frac{n^2}{2}.$
\end{theorem}
Now, we present the following lemma which follows from Tur\'an's theorem to show that a graph with few edges has a large independence number and hence gives us an upper bound for the $(a,b)$-domination number of $G$.
\begin{lemma}
Let $G$ be a graph with $n$ vertices and at most $\alpha n$ edges, where $\alpha>0$ is a fixed number. Then, the independence number of $G$ is at least $\frac{n}{2\alpha+1}$. 
\label{lem}
\end{lemma}

Our first application of this idea is the following theorem. The first part of the theorem was stated in Henning and Yeo in \cite{hen}  before, but here we give a much simpler proof. 

\begin{theorem} If $G$ is a graph with  $\delta(G)\geq 3$ and $n$ vertices, then $\gamma_{2,2}(G)\leq \frac{6}{7} n$. If $\delta(G)\geq 4$, then $\gamma_{2,2}(G)\leq \frac{4}{5} n$.
\label{t1}
\end{theorem}

\begin{proof}
We construct a graph $G'$ from $G$ as follows. The vertices of $G'$ are the same as $G$. When $\delta(G)\geq 3$, for each vertex $v$, choose three of its neighbors and join each pair of them with three edges (see Figure \ref{fig1}). This way we get a graph of at most $3n$ edges. In any independent set $A$ of $G'$, for each vertex $v$, there exists at least $2$ neighbors of $v$ not in $A$, for example in Figure \ref{fig1} at most one of the vertices $u$, $w$ or $t$ can be in $A$.
Therefore, $V\setminus A$ is a $(2,2)$-dominating set for $G$.  

According to Lemma \ref{lem}, the independence number of $G'$ is at least $\frac{n}{7}$. So, $\gamma_{2,2}(G)\leq \frac{6}{7}n$.

When $\delta(G)\geq 4$, for each vertex $v$ choose $4$ of its neighbors arbitrarily and join them with $2$ disjoint edges (see Figure \ref{fig2}). 
This way we get a graph $G'$ of at most $2n$ edges.

 For any independence set $A$ of $G'$,  $V(G)\setminus A$ is a $(2,2)$-dominating set for $G$,  because for any vertex $v$, at least two of its $4$ chosen neighbors are not inside $A$.  By Lemma \ref{lem}, the independence numebr of 
$G'$ is at least $\frac{n}{5}$ and hence $\gamma_{2,2}(G)\leq \frac{4}{5}n$.

\end{proof}

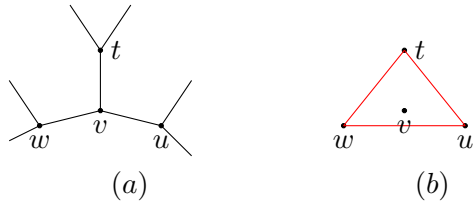
\begin{figure}[htpb]
	\centering
	\begin{tikzpicture}[scale=0.2]

\draw(0,0)--(4,-1);
\draw(0,0)--(-4,-1);
\draw(0,0)--(0,4);

\draw(0,0) node[below] {$v$};
\draw(4,-1) node[below] {$u$};
\draw(-4,-1) node[below] {$w$};
\draw(0,4) node[right] {$t$};

\filldraw(0,0) circle(4pt);
\filldraw(4,-1) circle(4pt);
\filldraw(-4,-1) circle(4pt);
\filldraw(0,4) circle(4pt);

\draw(4,-1)--(6,2);
\draw(4,-1)--(6,-3);

\draw(-4,-1)--(-6,2);
\draw(-4,-1)--(-6,-2);

\draw(0,4)--(2,7);
\draw(0,4)--(-2,7);

\filldraw(20,0) circle(4pt);
\filldraw(24,-1) circle(4pt);
\filldraw(16,-1) circle(4pt);
\filldraw(20,4) circle(4pt);

\draw(24,-1)--(16,-1)[red];
\draw(16,-1)--(20,4)[red];
\draw(20,4)--(24,-1)[red];

\draw(20,0) node[below] {$v$};
\draw(24,-1) node[below] {$u$};
\draw(16,-1) node[below] {$w$};
\draw(20,4) node[right] {$t$};

\draw(0,-5) node[right] {$(a)$};
\draw(20,-5) node[right] {$(b)$};

\end{tikzpicture}
\caption{The construction of $G'$ from $G$ for the case where $\delta(G)\geq 3$. (a) For each vertex $v$ in $G$, (b) we connect it's neighbors to each other in $G'$.}
\label{fig1}
\end{figure}

\begin{figure}[htpb]
	\centering
	\begin{tikzpicture}[scale=0.3]

\draw(0,0)--(4,-2);
\draw(0,0)--(-4,-2);
\draw(0,0)--(4,2);
\draw(0,0)--(-4,2);
\draw(0,0)--(0,3);

\draw(0,0) node[below] {$v$};
\draw(4,-2) node[below] {$u$};
\draw(-4,-2) node[below] {$w$};
\draw(4,2) node[right] {$t$};
\draw(-4,2) node[right] {$l$};
\draw(0,3) node[above]{$m$};

\filldraw(0,0) circle(4pt);
\filldraw(4,-2) circle(4pt);
\filldraw(-4,-2) circle(4pt);
\filldraw(4,2) circle(4pt);
\filldraw(-4,2) circle(4pt);
\filldraw(0,3) circle(4pt);

\draw(24,-2)--(16,-2)[red];
\draw(24,2)--(16,2)[red];

\draw(20,0) node[below] {$v$};
\draw(24,-2) node[below] {$u$};
\draw(16,-2) node[below] {$w$};
\draw(24,2) node[right] {$t$};
\draw(16,2) node[right] {$l$};
\draw(20,3) node[above]{$m$};

\filldraw(20,0) circle(4pt);
\filldraw(24,-2) circle(4pt);
\filldraw(16,-2) circle(4pt);
\filldraw(24,2) circle(4pt);
\filldraw(16,2) circle(4pt);
\filldraw(20,3) circle(4pt);

\draw(0,-5) node[right] {$(a)$};
\draw(20,-5) node[right] {$(b)$};

\end{tikzpicture}
\caption{The construction of $G'$ from $G$ for the case where $\delta(G)\geq 4$. (a) For each vertex $v$ in $G$, (b) we add two pairwise edges between it's four neighbors in $G'$.}
\label{fig2}
\end{figure}
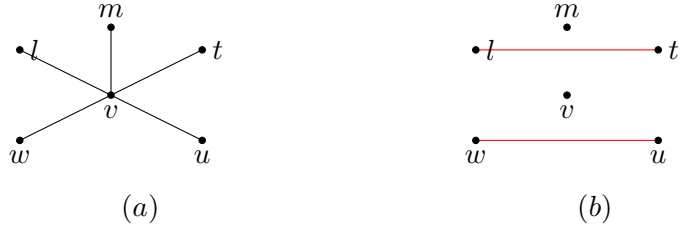

Note that for $\delta\geq 3$, Hening and Yeo conjectured that if $G\neq G_{14}$, then $\gamma_{2,2}(G)\leq \frac{5}{6}n$. By the Theorem \ref{3} we prove an improved result of their  conjecture for graphs with $\delta\geq 4$.

This idea can be applied for $k$-tuple total dominating number as well.

\begin{theorem} If $G$ is a graph with minimum degree at least $k+1$ and $n$ vertices, then $\gamma_{k,k}(G)\le \frac{k(k+1)}{k(k+1)+1}  n$ and if the minimum degree is at least $2k$ then $\gamma_{k,k}(G)\le \frac{2k}{2k+1} n$.
\label{t2}
\end{theorem}

\begin{proof}
Similar to the previous case, for each vertex $v$, we choose $k+1$ neighbors and connect every pair of them with $\frac{k(k+1)}{2}$ edges. The union of all these edges makes a graph $G'$ with at most $\frac{k(k+1)}{2} n$  edges, so by  Lemma \ref{lem} 
$G'$ has an independent set $A$ of size at least $\frac{1}{k(k+1)+1} n$ and hence $V(G)\setminus A$, forms a $(k,k)$-dominating set of size at most $\frac{k(k+1)}{k(k+1)+1} n$.

When $\delta\ge 2k$, for each vertex choose $2k$ of its neighbors and join them with $k$ disjoint edges. The union of all these edges makes a graph $G'$ with at most $kn$ edges and hence by  Lemma \ref{lem} $G'$ has an independent set $A$ of size at least $\frac{1}{2k+1} n$ and $V(G\setminus A)$ is a $(k,k)$-dominating set of size at most $\frac{2k}{2k+1} n$.
\end{proof}

In the following we generalize Theorem \ref{t1} and Theorem \ref{t2}.
When $\delta(G)=k+1+d$ for $0\leq d\leq k-1$, we construct the graph $G'$ as follows. For each vertex $v\in V(G)$, choose $k+1+d$ of its neighbors and partition them into $d+1$ parts and in each part
connect each pair of vertices by edges. Similar to Tur\'an's Theorem, one can make the parts almost of equal size so that the number of edges for each vertex $v$ is at most ${k+1+d\choose 2 }-(1-\frac{1}{d+1})\frac{(k+1+d)^2}{2}$ which is  $\frac{k(k+d+1)}{2d+2}$. So, we will get the following theorem

\begin{theorem} If $G$ is a graph with minimum degree $\delta=k+1+d$ for $0<d<k-1$, then $\gamma_{k,k}(G)\le \frac{2d+(k-d)(k-d+1)}{2d+(k-d)(k-d+1)+1} n$.
\end{theorem}

\begin{theorem} If $G$ is a graph with $n$ vertices and minimum degree $\delta\ge 3$ such that at least half of vertices have degree at least $4$, then $\gamma_{2,2}(G)\le \frac{5}{6} n$.
\end{theorem}

\begin{proof} The graph $G'$ is built by vertices of $G$ and for any vertex of degree $3$ all three neighbors are joined by three edges and for a vertex of degree bigger than $3$, only two disjoint edges between two pairs of neighbors are added. As before a complement of an independent set for $G'$ is a $(2,2)$-dominating set for $G$. $G'$ has at most $3n-\frac{1}{2} n= \frac{5}{2} n$ edges. So, by Lemma \ref{lem}, $G'$ has an independent set  $A$ of size at least $\frac{n}{6}$ and therefore $V(G)\setminus A$ has size at most $\frac{5}{6} n$.
\end{proof}

\begin{theorem} 
 Let $a<b$ and $G$ is a graph with minimum degree $\delta\ge a+b$, then $\gamma_{a,b}(G)\le \frac{2b}{2b+1}n$. If $G$ has a spanning regular subgraph of degree $b-a$, then we have $\gamma_{a,b}(G)\le \frac{a+b}{a+b+1} n$. In particular if $G$ is a graph with $\delta\ge 2k-1$ and $G$ has a perfect matching then the $k$-tuple dominating number  $\gamma_{k-1,k}(G)\le \frac{2k-1}{2k} n$.
 \end{theorem}
 
 \begin{proof} The construction of our graph $G'$ is as follows. The vertices of $G'$ are the same as $G$. For any vertex $v$ of $G$ we add $a$ disjoint edges between $a$ pairs of neighbors of $v$ and add $b-a$ edges adjacent to $v$ with end-points different from the chosen $a$ pairs of the neighbors in the first step. Hence we have $b$ edges associated to each vertex and $G'$ has at most $bn$ vertices. If $A$ is an independent set for $G'$. For any vertex $v\in A$ all the $b-a$ edges of $G'$ adjacent to $v$ have end-points outside $A$ and at least $a$ vertices of the $a$ disjoint edges beween neigbors of $v$ are outside $A$. Hence, $v$ has at least $b$ neighbors outside $A$. If $v$ is not in $A$, then we have at least $a$ vertices among neighbors of $v$ that do not belong to $A$. Therefore, the complement of $A$ is an $(a,b)$-dominating set for $G$. By  Lemma \ref{lem}, a graph with $bn$ edges has an independent set of size at least $\frac{1}{2b+1} n$ vertices and the first part of the theorem is proved. If $G$ has an spanning $(b-a)$-regular subgraph, then one can use that subgraph for the $b-a$ edges for each vertex adjacent to that vertex and therefore the number of edges of $G'$ is at most $(a+\frac{b-a}{2}) n= \frac{a+b}{2} n$ and  the independent set has a size of at least $\frac{1}{a+b+1} n$ vertices and the second part is proved as well. To prove the last part, notice that existence of a spanning $1$-regular subgraph of $G$ is equivalent to existence of a perfect matching for $G$.
 \end{proof}

\section{Upper bound for $(a,b)$-domination number using Lov\'asz local lemma}
In this section, we give a probabilistic method to compute upper bounds for $(a,b)$-domination number of graphs. Our idea is to choose some vertices randomly as the $(a,b)$-dominating set of $G$. Using this method we can improve previous results in some special cases.

\begin{theorem}
For any $\frac{1}{2} \leq \alpha<1$, there is $r_0>0$ such that, if $r\geq r_0$ and $G$ is an $r$-regular graph,  $\gamma_{a,b}(G)\leq \alpha n$. 
\end{theorem}

Suppose that we use $N$ colors to color the vertices of $G$, such that each vertex $v_i$ has one of the colors of $c_1,c_2,..,c_N$ with probability $1/N$. Then, there are at least $n/N$ of the vertices with the same color, suppose red, we select all the vertices except vertices with color red as a possible $(a,b)$-dominating set. Now, we want to show that the probability that the selected vertices form an $(a,b)$-dominating set is greater that $0$, which means there is a coloring that the selected vertices form an $(a,b)$-dominating set for $G$.

Suppose that we choose all the colors except red as an $(a,b)$-dominating set, then if all the vertices with color red
have at least $b$ neighbors with colors other than red and all the vertices with color other than red have at least $a$ neighbors
with colors other than red, this means that each selected vertex has at least $a$ neighbors that are selected and each vertex that is not selected has at least $b$ neighbors that are selected. This implies that the selected vertices form an $(a,b)$-dominating set.

Now we want to show that there exists an integer number $N$ such that if we randomly color the vertices of $G$ with $N$ colors, then there exists a coloring of vertices that 
if we select all the vertices except the vertices with an arbitrary color, this selected set is an $(a,b)$-dominating set.
Note that we wish to minimize $N$.
First, we present Lov\'asz local lemma,
\begin{lemma}
Let $A_1, A_2,..., A_n$ be a sequence of events such that each event occurs with probability at most $p$ and such that each event is independent of all the other events except for at most $d$ of them. If $epd\leq 1$,
then there is a nonzero probability that none of the events occurs.
\end{lemma}

According to our coloring, let $A_i$ be an event such that if $v_i$ is not selected in our dominating set, then $v_i$ has less than $b$ neighbors with color other than the color of $v_i$ and if $v_i$ is selected in our dominating set, then $v_i$ has less than $a$ neighbors with color other than the color of vertices that are not selected in our dominating set. Then, for each $A_i$

$$P=P(A_i)= N\frac{1}{N}  \frac{\sum^{b-1}_{j=0}{ r\choose j} (N-1)^j }{N^r}+ N\frac{N-1}{N}  \frac{\sum^{a-1}_{j=0} {r\choose j} (N-1)^j }{N^r}$$ 

On the other hand, each event $A_i$ depends on it's neighbors and neighbors of it's neighbors. Which means $A_i$ depends on at most $r^2$ of the events. So, if
$$eP r^2\leq 1,$$ then there is a nonzero probability that none of the events occurs, which means all the  selected vertices have at least $a$ neighbors in the selected vertices and the vertices that are not selected have at least $b$ neighbors in the selected vertices. In Table \ref{com}, we have calculated the minimum value of $N$ to compute upper bounds for various values of 
$r$, $a$ and $b$.
For a given $N\geq 2$, the $ePr^2$ is a fraction that its numerator is a polynomial in $r$ and its denumerator is exponential in $r$. So, there exists  $r_0>0$ such that if $r\geq r_0$, then  $ePr^2<1$.

Note that in our computation we have assumed that the graph $G$ is regular. It can be easily seen that for arbitrary graphs with minimum degree $\delta$ and maximum degree $\Delta$, in calculating $P(A_i)$ instead of $r$ we can use $\delta$ and each $P(A_i)$ depends on at most $\Delta^2$ of the vertices. So, for the cases when the graph is not regular we can use this formula again.

\begin{table}[h]
\caption{Upper bounds for $\gamma_{a,b}(G)$ achieved by  Lov\'asz local lemma for various values of $a$, $b$, $\delta$ and $\Delta$}
\begin{center}
\begin{tabular}{|c|c|c|c|c|}
\hline
$\delta$& $\Delta$& $a$& $b$& Upper bound \\
 \hline
7&7&2&2&$\frac{3}{4}n$\\
\hline
7&8&2&2&$\frac{3}{4}n$\\
\hline

9&9&2&2&$\frac{2}{3}n$\\
\hline
9&10&2&2&$\frac{2}{3}n$\\
\hline
9&11&2&2&$\frac{2}{3}n$\\
\hline

14&14&2&2&$\frac{1}{2}n$\\
\hline
8&8&1&2&$\frac{2}{3}n$\\
\hline

8&9&1&2&$\frac{2}{3}n$\\
\hline
8&10&1&2&$\frac{2}{3}n$\\
\hline
8&11&1&2&$\frac{2}{3}n$\\
\hline
13&13&1&2&$\frac{1}{2}n$\\
\hline
13&14&1&2&$\frac{1}{2}n$\\
\hline

8&8&2&1&$\frac{2}{3}n$\\
\hline
13&13&2&1&$\frac{1}{2}n$\\
\hline
13&14&2&1&$\frac{1}{2}n$\\

\hline
\end{tabular}
\end{center}
\label{com}
\end{table}

\section{Conclusion}
In this paper we present improved upper bounds for $(a,b)$-domination number of graphs. 
For a given graph $G$, we defined graph $G'$ such that removing any independent set of $G'$ gives an $(a,b)$-dominating set for $G$. We use  Tur\'an's Theorem to compute lower bound for the independence number of $G'$. 
Next, we present a randomized approach to compute upper bound of $(a,b)$-domination number that uses Lov\'asz local lemma.
In some cases we improve the previous results and also we present some new upper bounds that had not been studied before as far as we know.
There are still open problems in this regard.

\end{document}